\documentclass{article}
\usepackage{amsmath, amsthm, amsfonts, amssymb, amscd, stmaryrd,xcolor,enumerate, enumitem}

\usepackage{graphicx,filecontents}
\usepackage{tikz}

\usepackage[
    backend=biber,
    style=numeric,
    maxbibnames=99
  ]{biblatex}
\renewbibmacro{in:}{%
  \ifentrytype{article}
    {}
    {\bibstring{in}%
     \printunit{\intitlepunct}}}
 \DeclareNameAlias{default}{family-given}

\theoremstyle{definition}

\theoremstyle{plain}
\newtheorem{Thm}{Theorem}
\newtheorem{Lem}[Thm]{Lemma}
\newtheorem{Cor}[Thm]{Corollary}

\newtheorem{Prop}[Thm]{Proposition}

\numberwithin{equation}{section}	
\numberwithin{figure}{Prob}
\numberwithin{table}{Prob}
\numberwithin{Thm}{section}

\newcommand{\Z}{\mathbb{Z}}

\newcommand{\C}{\mathbb{C}}

\renewcommand{\S}{\mathfrak{S}}

\newcommand{\Span}{\text{Span}}

\title{Almost Commutative Terwilliger Algebras I: The Group Association Scheme}
\author{Nicholas L. Bastian, Stephen P. Humphries}

\date{\today}

\begin{document}

\maketitle

\begin{abstract}
Terwilliger algebras are a subalgebra of a matrix algebra constructed from an association scheme. Rie Tanaka defined what it means for a Terwilliger algebra to be almost commutative and gave five equivalent conditions. In this paper we first determine an equivalent sixth condition for a Terwilliger algebra coming from a commutative Schur ring to be almost commutative. We then provide a classification of which finite groups result in an almost commutative Terwilliger algebra when looking at the group association scheme determined by the conjugacy classes. In particular, we show that all such groups are either abelian, or Camina groups. We then compute the dimension of each Terwilliger algebra, and we also express each of the group association schemes with an almost commutative Terwilliger algebra as a wreath product of the group schemes of finite abelian groups and $1-$class association schemes. Furthermore, we give the non-primary primitive idempotents for each Terwilliger algebra for those groups.     \end{abstract}



\section{Introduction}

Terwilliger algebras were originally developed in the 1990's by Paul Terwilliger to study commutative association schemes. In particular, he looked at P-Polynomial and Q-Polynomial association schemes. Over the course of the three papers \cite{Terwilliger1,Terwilliger2,Terwilliger3} in which Terwilliger algebras were originally introduced, Terwilliger finds a combinatorial characterization of thin P- and Q-Polynomial association schemes. Since that time Terwilliger algebras have been studied for many types of association schemes. One of the most common is the group association scheme (\cite{GroupAssoc1, Bannaiarticle, Bastian, GroupAssoc2}). 

Tanaka \cite{Tanaka} determined a situation in which the Terwilliger algebra has a particular Wedderburn structure (see Theorem \ref{thm:tanaka}) called almost commutative. One of Tanaka's results is that when a Terwilliger algebra is almost commutative, the corresponding association scheme can be written as a wreath product.

In this paper the group $G$ will always be finite and $Q_8$ will be the quanterion group with $8$ elements. For $x\in G$, we let $x^G$ be the conjugacy class of $x$. We let $T(G)$ denote the Terwilliger algebra for the group association scheme, denoted $\mathcal{G}(G)$, of $G$ with base point $e$. We let $\mathcal{K}_n$ denote the trivial association scheme for a set of size $n$. The main result is as follows:

\begin{Thm}\label{thm:main}
         $T(G)$ is almost commutative if and only if $G$ is isomorphic to one of the following:
        \begin{enumerate}[label=(\roman*)]
        \itemindent=-8pt
        \item An abelian group. In this case $\dim T(G)=|G|^2$ and there is a single Wedderburn component of dimension $|G|^2$.
        \item The Frobenius group $\Z_p^r\rtimes \Z_{p^{r}-1}$, for prime $p$ and $r>0$. In this case $\dim T(G)=p^{2r}+p^r-1$ and $\mathcal{G}(G)=\mathcal{K}_{p^r}\wr \mathcal{G}(\Z_{p^r-1})$.
        \item A Camina $p-$group of order $p^n$ with nilpotency class $2$ for prime $p$. If $|Z(G)|=p^k$, then \[\dim T(G)=(p^{n-k}+p^k-1)(p^{n-k}+p^k-2)+p^n.\]
        \hspace{-0.4cm} Also $\mathcal{G}(G)=\mathcal{G}(Z(G))\wr \mathcal{G}(G/Z(G))$.
        \item A Camina $p-$group with nilpotency class $3$ for prime $p$. If $|Z(G)|=p^k$ and $|G'\colon Z(G)|=p^{n}$, then \[\dim T(G)=3p^{k+n}+3p^{k+2n}+p^{2k}-6p^k+p^{4n}+3p^{3n}-5p^{2n}-6p^n+7.\] 
        \hspace{-0.4cm} Also $\mathcal{G}(G)=\mathcal{G}(Z(G))\wr \mathcal{G}(G'/Z(G))\wr \mathcal{G}(G/G')$.
        \item The group $\Z_3^2\rtimes Q_8$. In this case $\dim T(G)=44$ and $\mathcal{G}(G)=\mathcal{K}_9\wr \mathcal{G}(\Z_2)\wr \mathcal{G}(\Z_2^2)$.
    \end{enumerate}
    \end{Thm}
    That the groups $(i)-(v)$ are the only possibilities follows from Theorems \ref{thm:sequiv} and \ref{thm:groupequiv}. 

    Theorem \ref{thm:main} will be used in determining a classification of all groups for which the Terwilliger algebra resulting from a strong Gelfand pair is almost commutative \cite{Bastiansgp}.

This paper is organized as follows. Section $2$ introduces Terwilliger algebras and defines almost commutative. The main result of this section is Theorem \ref{thm:sequiv}, which is a sixth equivalent condition for the Terwilliger algebra coming from a commutative Schur ring to be almost commutative. In Section $3$ we discuss Camina groups. In Section $4$ we will give the classification of almost commutative Terwilliger algebras for group association schemes. We also find $\dim T(G)$ in each case. Then in Section $5$ we determine how to write each of the group schemes having an almost commutative Terwilliger algebra as a wreath product of simpler association schemes. After doing this we find the Wedderburn components of each of these Terwilliger algebras using results from \cite{WreathProduct1}.\\
\textbf{Acknowledgment:} All computations made in the preparation of this paper were performed using Magma \cite{Magma}.

\section{Terwilliger Algebras}

Let $\Omega$ be a finite set. Suppose $A_0,A_1,\cdots, A_d\in M_{|\Omega|}(\C)$ are $0,1$ matrices with rows and columns indexed by the elements of $\Omega$. Let $A^t$ denote the transpose of the matrix $A$. If additionally 
\begin{enumerate}[leftmargin=*]
    \item $A_0=I_{|\Omega|}$;
    \item for all $i=1,2,\dots, d$, we have $A_i^t=A_j$ for some $j=1,2,\dots ,d$;
    \item for all $i,j\in \{1,2,\dots, d\}$ we have $A_iA_j=\sum_{k=0}^d p_{ij}^k A_k$;
    \item none of the $A_i$ is equal to $0_{|\Omega|}$ and $\sum_{i=0}^d A_i$ is the all $1$ matrix;
\end{enumerate}
then we say that $\mathcal{A}=(\Omega,\{A_i\}_{0\leq i\leq d})$ is an \emph{association scheme}. We call $A_0,A_1,\cdots, A_d$ the \emph{adjacency matrices} of the association scheme and the constants $p_{ij}^k$ the \emph{intersection numbers}. If $A_iA_j=A_jA_i$ for all $i,j$, we have a \emph{commutative association scheme}. 

Let $\mathcal{X}=(X,\{A_i\}_{0\leq i\leq c})$ and $\mathcal{Y}=(Y,\{B_i\}_{0\leq j\leq d})$ be two association schemes where $|X|=m$ and $|Y|=n$. Then the \emph{wreath product} of $\mathcal{X}$ and $\mathcal{Y}$, denoted $\mathcal{X}\wr \mathcal{Y}$, is defined on the set $X\times Y$ with adjacency matrices
\[D_0=B_0\otimes A_0,\ D_k=I_n\otimes A_k,\ (\text{for $1\leq k\leq c$}),\ D_{c+k}=B_k\otimes J_m,\ (\text{ for $1\leq k\leq d$}).\]

The association schemes that we will be interested in correspond to Schur rings. Let $R$ be a commutative ring with $1$. For $C\subseteq G$, let $\overline{C}=\sum_{g\in C} g\in R[G]$ and $C^*=\{g^{-1}\colon g\in G\}$.

Let $\mathcal{P}$ be a partition of $G$. We say $\mathfrak{S}=\Span_R\{\overline{C}|C\in \mathcal{P}\}$ is a \emph{Schur ring} (see \cite{Ma89,Wielandt49}) if:
\begin{enumerate}[leftmargin=*]
    \item $\{e\} \in \mathcal{P};$
    \item if $C \in \mathcal{P}$, then $C^* \in \mathcal{P};$
    \item for all $C, D \in \mathcal{P}$,
$\overline{C}\cdot\overline{D} = \sum_{E\in \mathcal{P}} \lambda_{CDE}\overline{E}$.
\end{enumerate}

We call the sets in $\mathcal{P}$ the \emph{principal sets} of $\S$ and the constants $\lambda_{CDE}$ the \emph{structure constants}.

Schur rings give rise to association schemes in the following way. We define $|G|\times |G|$ matrices $A_i$, $0\leq i\leq d$, by
\[(A_i)_{xy}=\left\{\begin{array}{cc}
    1 & \text{ if } yx^{-1}\in P_i  \\
    0 & \text{ otherwise.} 
\end{array}\right.\]
With this construction $(G,\{A_i\}_{0\leq i\leq d})$ is an association scheme. If the Schur ring is commutative then the association scheme is commutative. The structure constants $\lambda_{ijk}$ coming from $\overline{P_i}\cdot \overline{P_j}=\sum_{P_k\in \mathcal{P}}\lambda_{ijk}\overline{P_k}$ are the intersection numbers of the association scheme. In particular, partitioning $G$ by the conjugacy classes of $G$, we get a Schur ring. This gives a commutative association scheme, called the \emph{group association scheme}, denoted $\mathcal{G}(G)$. 

Returning to arbitrary association schemes, we note that $\mathfrak{A}=\Span_\C\{A_i\colon 0\leq i\leq d\}$ is an algebra, called the \emph{Bose-Mesner algebra} of the association scheme. We let $E_0,E_1,\dots, E_d$ be the primitive idempotents of $\mathfrak{A}$. As $\mathfrak{A}$ is closed under \emph{Hadamard product}, denoted $\circ$, we have $E_i\circ E_j\in \mathfrak{A}$. Then there are $q_{ij}^k\in \C$, called the \emph{Krein parameters}, such that
\[E_i\circ E_j=\frac{1}{|\Omega|}\sum_{k=0}^d q_{ij}^kE_k.\]

Fix $x\in \Omega$. We further define diagonal matrices $E_i^*(x)$ \emph{with base point $x$} by the rule
\[(E_i^*(x))_{yy}=\Bigg\{ \begin{array}{cc}
    1  & (x,y)\in R_i \\
    0 & \text{otherwise}.
\end{array}\]
We call $\mathfrak{A}^*(x)=\Span_\C\{E_i^*(x)\colon 0\leq i\leq d\}$ the \emph{dual Bose-Mesner Algebra} of $\mathcal{A}$ with respect to $x$.  

Now we define the \emph{Terwilliger algebra} with base point $x\in \Omega$ of the association scheme $\mathcal{A}=(\Omega,\{A_i\}_{0\leq i\leq d})$ to be the subalgebra of $M_{|\Omega|}(\C)$ generated by $\mathfrak{A}$ and $\mathfrak{A}^*(x)$. We denote this by $T(x)$. For the group association scheme we always take the base point $x$ to be the identity $e\in G$. We let $T(G)$ denote the Terwilliger algebra for $\mathcal{G}(G)$.

Let $T_0(x)=\Span_\C \{E_i^*(x)A_jE_k^*(x)\colon 0\leq i,j,k\leq d\}$.
We note $\dim T_0(x)=|\{(i,j,k)\colon p_{ij}^k\neq 0\}|$ \cite{Bannaiarticle}. A Terwilliger algebra is \emph{triply regular} if $T(x)=T_0(x)$.

If $|\Omega|>1$, then the Terwilliger algebra with base point $x\in \Omega$, is non-commutative and semi-simple for all $x\in \Omega$ \cite{Terwilliger1}. As such, it has a Wedderburn decomposition. Given an association scheme $(\Omega,\{A_i\}_{0\leq i\leq d})$, the Terwilliger algebra for this association scheme with base point $x$ always has an irreducible ideal $V$ of dimension $(d+1)^2$, called the \emph{primary component} (see \cite{WreathProduct1}).

Closely related to the Wedderburn decomposition of the Terwilliger algebra are those $T(x)-$modules that are irreducible when $T(x)$ acts on $\C^{|\Omega|}$ by left multiplication. Then $\C^{|\Omega|}$ decomposes into a direct sum of irreducible $T(x)-$ modules. Letting $\hat{x}\in \C^{|\Omega|}$ have $1$ in position $x$ and $0$ elsewhere, it is shown in \cite{Terwilliger1} that $\mathfrak{A}\hat{x}$ is an irreducible $T(x)-$module of dimension $d+1$. We call this irreducible module the \emph{primary module}.

A Terwilliger algebra $T(x)$ is \emph{almost commutative} (AC) if every non-primary irreducible $T(x)-$module is $1-$ dimensional. There is a classification of such Terwilliger algebras when the underlying association scheme is commutative:

\begin{Thm}[Tanaka, \cite{Tanaka}]\label{thm:tanaka}  Let $\mathcal{A}=(\Omega,\{R_i\}_{0\leq i\leq d})$ be a commutative association scheme. Let $T(x)$ be the Terwilliger algebra of $\mathcal{A}$ for some $x\in \Omega$. The following are equivalent:
    \begin{enumerate}[leftmargin=*]
        \item Every non-primary irreducible $T(x)-$module is $1-$ dimensional for some $x\in \Omega$.
        \item Every non-primary irreducible $T(x)-$module is $1-$ dimensional for all $x\in \Omega$.
        \item The $p_{ij}^h$ satisfy: for all distinct $h,i$ there is exactly one $j$ such that $p_{ij}^h\neq 0$ $(0\leq h,i,j\leq d)$.
        \item The $q_{ij}^h$ satisfy: for all distinct $h,i$ there is exactly one $j$ such that $q_{ij}^h\neq 0$ $(0\leq h,i,j\leq d)$.
        \item $\mathcal{A}$ is a wreath product of association schemes $\mathcal{A}_1,\mathcal{A}_2,\cdots, \mathcal{A}_n$ where each $\mathcal{A}_i$ is either a $1-$class association scheme or the group scheme of a abelian group.
    \end{enumerate} 
Moreover, a Terwilliger algebra satisfying these equivalent conditions is triply regular.
\end{Thm}

The focus of this paper is to determine exactly when $T(G)$ is AC and then study the resulting Terwilliger algebra. We now give a condition that will show when $T(G)$ is AC.

 \begin{Thm}\label{thm:sequiv}
    Let $\S$ be any commutative Schur ring for $G$ with principal sets $P_i$, $0\leq i\leq d$. Let $T(G,\S)$ be the Terwilliger algebra of $G$ with respect to $\S$. Then $T(G,\S)$ is AC if and only if $\S$ has the following property: for all $x,y\in G$ if $x\in P_i$, $y\in P_j$, $xy\in P_h$, and $P_i\neq P_j^*$ then $P_iP_j=P_h$. 
\end{Thm}

\begin{proof}
    First suppose that $T(G,\S)$ is AC. Let $x,y\in G$ with $x\in P_i$, $y\in P_j$ and $P_i\neq P_j^*$. Since $P_j\neq P_{i}^*$, and $T(G,\S)$ is AC, there is a unique $h$ with $p_{i'h}^j\neq 0$, so $P_j\subseteq P_{i'}P_h$. Then for $z\in P_j$, there are $a\in P_{i'}$ and $b\in P_h$ such that $z=ab$. Then $b=a^{-1}z$, so $b\in P_iP_j$. Thus, $P_h\subseteq P_iP_j$. Now suppose that $P_m\subseteq P_iP_j$. Then for $a\in P_m$, there is $b\in P_i$ and $c\in P_j$ such that $a=bc$. Then $c=b^{-1}a$, so $c\in P_{i'}P_m$. Thus $P_j\subseteq P_{i'}P_m$. This implies $p_{i'm}^j\neq 0$. By the uniqueness of $p_{i'h}^j$, we have $P_m=P_h$. So $P_h$ is the only principal set in $P_iP_j$. We know $\overline{P_i}\cdot\overline{P_j}=\sum_{\ell} p_{ij}^\ell \overline{P_\ell}$. For all $p_{ij}^\ell\neq 0$, we have $P_\ell\subseteq P_iP_j$. Since $P_h$ is the unique principal set in $P_iP_j$ we must have $P_iP_j=P_h$. Note $x\in P_i$ and $y\in P_j$, so $xy\in P_h$ as $P_iP_j=P_h$. Therefore, $P_h$ is the principal set containing $xy$ as claimed. 

   Conversely, suppose that for all $x,y\in G$, if $x\in P_i$, $y\in P_j$, $xy\in P_h$, and $P_i\neq P_j^*$ then $P_iP_j=P_h$. Let $P_i$ and $P_h$ be distinct principal sets with $x\in P_i$ and $z\in P_h$. Let $P_j$ be the principal set containing $x^{-1}z$. Notice $z\in P_iP_j$, so $P_h\subseteq P_iP_j$ which implies $p_{ij}^h\neq 0$. So we have found a $j$ such that $p_{ij}^h\neq 0$. Now suppose that $P_w$ is any principal set such that $p_{iw}^h\neq 0$. This implies that $P_h\subseteq P_iP_w$. Then $z=ab$ for some $a\in P_i$ and $b\in P_w$, so $e=z^{-1}ab$. Notice that $z^{-1}\in P_h^*$, $a\in P_i$, and $b\in P_w$, so $e\in P_h^*P_iP_w$. As $P_i\neq P_h$, $z^{-1}\in P_h^*$, $x\in P_i$, and $z^{-1}x\in P_j^*$ we have $P_h^*P_i=P_j^*$.  Hence, $e\in P_h^*P_iP_w= P_j^*P_w$. Then there exists a $g\in P_j^*$ such that $g^{-1}\in P_w$. However, $g^{-1}\in P_j$ and so $P_w=P_j$. So for distinct $P_i,P_h$ there is a unique $j$ such that $p_{ij}^h\neq 0$. By Theorem \ref{thm:tanaka}, $T(G,\S)$ is AC.
\end{proof}

\section{Camina Groups}
A group $G$ is a \emph{Camina group} (\cite{Camina,Lewis1}) if every conjugacy class of $G$ outside of $G'$ is a coset of $G'$. We will need:

\begin{Thm}[Dark and Scoppola, \cite{Dark1}] Let $G$ be a Camina group. Then one of the following is true:
        \begin{enumerate}[leftmargin=*]
            \item $G$ is a Frobenius group whose Frobenius complement is cyclic.
            \item $G$ is a Frobenius group whose Frobenius complement is $Q_8$.
            \item $G$ is a $p-$group for some prime $p$ of nilpotency class $2$ or $3$.
        \end{enumerate}
\end{Thm}

\begin{Prop}[Macdonald, \cite{MacDonald}]\label{thm:Camina3}
        Let $G$ be a Camina $p-$group of nilpotency class $3$. Then the lower central series of $G$ is $\{e\}\leq Z(G)\leq G'\leq G$. Furthermore, $|G\colon G'|=p^{2n}$ and $|G'\colon Z(G)|=p^n$ for some even integer $n$. Additionally, $G/G'$ and $G'/Z(G)$ are elementary abelian groups.
\end{Prop}

\begin{Prop}[Macdonald, \cite{MacDonald}]\label{prop:cam2center}
        Let $G$ be a Camina p-group. Then $Z(G)$ is elementary abelian.
    \end{Prop}

\begin{Prop}[Lemma 2.3, \cite{Lewis1}]\label{prop:cammod}
    Let $G$ be a Camina group.
    \begin{enumerate}[leftmargin=*]
        \item If $N$ is a normal subgroup of $G$, then either $N\leq G'$ or $G'\leq N$.
        \item If $N<G'$ is a normal subgroup of $G$, then $G/N$ is also a Camina group.
    \end{enumerate}
\end{Prop}

\begin{Prop}[Macdonald, Theorem 5.2(i), \cite{MacDonald}]\label{lem:pcamclass}
    Let $G$ be a Camina $p-$group. Let $\gamma_2(G)=[G,G]$ and $\gamma_3(G)=[\gamma_2(G),G]=Z(G)$ be the second and third terms in the lower central series of $G$. Then
        \[x^G=x\gamma_2(G)\text{ if }x\in G\setminus \gamma_2(G),\ x^G=x\gamma_3(G)\text{ if }x\in \gamma_2(G)\setminus \gamma_3(G),\ x^G=\{x\} \text{ if }x\in \gamma_3(G).\]
\end{Prop}

\begin{Thm}[Cangelmi and Muktibodh, \cite{Few_Conj}]\label{thm:twoclass}
     Let $G$ be a group. Then, $G$ is a Camina group and $G'$ is the union of two conjugacy classes if and only if either $G$ is a Frobenius group with Frobenius kernel $\Z_p^r$ and Frobenius complement $\Z_{p^r-1}$, or $G$ is an extra-special $2-$group.
\end{Thm}

\section{AC Terwilliger Algebras for Group Association Schemes}

For the remainder of this paper let $C_0=\{e\},C_1,C_2,\cdots, C_d$ be the conjugacy classes of $G$. We show:

\begin{Thm}\label{thm:acclassify}
        Let $G$ be a group. Then $T(G)$ is AC if and only if $G$ is isomorphic to one of the following groups: $(i)$ an abelian group; $(ii)$ the Frobenius group $\Z_p^r\rtimes \Z_{p^{r}-1}$, for prime $p$ and $r>0$; $(iii)$ a Camina $p-$group, for prime $p$; $(iv)$ the group $\Z_3^2\rtimes Q_8$.
    \end{Thm}

    We note that Theorem \ref{thm:acclassify} is part of Theorem \ref{thm:main}. The proof of this theorem immediately follows from Theorem \ref{thm:sequiv} and the following result by Dade and Yadav.

    \begin{Thm}[\cite{Product_Conj}]\label{thm:groupequiv} A  group $G$ satisfies the property that for all $x,y\in G$ such that $x^G\neq (y^{-1})^G$, $x^Gy^G=(xy)^G$ if and only if $G$ is isomorphic to exactly one of the following groups: $(i)$ an abelian group; $(ii)$ the Frobenius group $\Z_p^r\rtimes \Z_{p^{r}-1}$, for prime $p$ and $r>0$; $(iii)$ a Camina $p-$group, for prime $p$; $(iv)$ the group $\Z_3^2\rtimes Q_8$.
    \end{Thm}

    Now we shall compute $\dim T(G)$ in each case $(i)-(iv)$. For abelian groups, we have:

\begin{Prop}[Bannai and Munemasa, \cite{Bannaiarticle}] Let $G$ be a  abelian group. Then $\dim T(G)=|G|^2$.
\end{Prop}

A direct computation using Magma \cite{Magma} yields:

\begin{Thm}\label{cor:72,41 dim}
        Let $G=\Z_3^2\rtimes Q_8$. Then $\dim T(G)=44$.
    \end{Thm}

         \begin{Thm}\label{cor:frobdim}
            Let $G=\Z_{p}^r\rtimes \Z_{p^r-1}$ for some prime $p$ and $r\geq 1$. Then $\dim T(G)=p^{2r}+p^r-1$.
        \end{Thm}

        \begin{proof}
         By Theorems \ref{thm:tanaka} and \ref{thm:acclassify}, $T(G)$ is AC and triply regular. Therefore, $T(G)=T_0(G)$ and $\dim T(G)=|\{(i,j,k)\colon p_{ij}^k\neq 0\}|$. 
         
         Since $G$ is a Camina group, its classes outside of $G'$ are all cosets of $G'$. By Theorem \ref{thm:twoclass}, the classes of $G$ inside $G'$ are $\{e\}$ and $G'\setminus \{e\}$. Note that $G'\cong \Z_{p}^r$ and $G/G'\cong \Z_{p^r-1}$. Then there are $p^r-2$ classes of $G$ outside of $G'$. So the number of classes of $G$ is $2+(p^r-2)=p^r$. Recall that
         \[\dim T_0(G)=|\{(i,j,k)\colon p_{ij}^k\neq 0\}|.\]
         We count the number of triples with $p_{ij}^k\neq 0$ by considering cases based on the types of classes. Let $C_0=\{e\}$ and $C_1=G'\setminus \{e\}$.\\
        {\bf Case 1: }$C_i=C_0$. Now $\overline{C_0}\cdot\overline{C_j}=\overline{C_j}$ for all $j$, thus there are $p^r$ triples $(0,j,j)$ (or $(j,0,j)$) such that $p_{0j}^j\neq 0$ (or $p_{j0}^j\neq 0$), giving a total of $2p^r$ triples that involve the class $C_0=\{e\}$ with $p_{ij}^h\neq 0$. This double counts the triple $(0,0,0)$. Hence, there are $2p^r-1$ nonzero triples that involve $C_0=\{e\}$.\\
        {\bf Case 2: }$C_i=C_1$. We consider $\overline{C_1}\cdot\overline{C_j}$, where $j>1$. Then $\overline{C_1}\cdot\overline{C_j}=\overline{(G'-e)}\cdot\overline{gG'}=|G'|\overline{gG'}-\overline{gG'}=(|G'|-1)\overline{gG'}=(p^r-1)\overline{C_j}$. So $(1,j,j)$ and $(j,1,j)$ are triples with $p_{1j}^j=p_{j1}^j\neq 0$. There are $p^r-2$ possible choices for $C_j$, so we have $2(p^r-2)$ triples. Now $\overline{C_1}\cdot\overline{C_1}=\overline{G'\setminus \{e\}}\cdot \overline{G'\setminus \{e\}}=|G'|\overline{G'}-2\overline{G'}+\{e\}=(p^n-2)\overline{G'\setminus \{e\}}+(p^n-1)\overline{\{e\}}$. Then the triples $(1,1,0)$ and $(1,1,1)$ also have nonzero $p_{11}^0$ and $p_{11}^1$. Thus, we get $2p^r-4+2=2p^r-2$ nonzero triples.\\
        {\bf Case 3: }$C_i=gG'$ where $g\not\in G'$. We only need to consider $C_j=hG'$. For $C_j\neq g^{-1}G'$, we have $C_iC_j=ghG'\neq G'$, so $\overline{C_i}\cdot\overline{C_j}=|G'|\overline{ghG'}=|G'|\overline{C_k}$ for some conjugacy class $C_k$. Thus, $(i,j,k)$ is the only triple with our choice of $i,j$ such that $p_{ij}^k\neq 0$ in this case. With $C_i$ fixed, we have a total of $p^r-2-1$ choices for $C_j$ such that $C_j\neq g^{-1}G'$. As each choice of $C_j$ gives a unique $C_k$ such that $p_{ij}^k\neq 0$ we get a total of $p^r-3$ triples $(i,j,k)$ in this case. If $C_j=g^{-1}G'$, then $\overline{C_i}\cdot\overline{C_j}=\overline{gG'}\cdot\overline{g^{-1}G'}=|G'|\overline{G'}=|G'|\overline{C_0}+|G'|\overline{C_1}$. Thus, the triples $(i,j,0)$ and $(i,j,1)$ have nonzero $p_{ij}^0$ and $p_{ij}^1$ in this case. We then have found a total of $p^r-1$ different $(i,j,k)$ such that $p_{ij}^k\neq 0$ for our fixed $C_i$. We have a total of $p^r-2$ choices for $C_i$, hence there are a total of $(p^r-1)(p^r-2)$ triples $(i,j,k)$ such that $p_{ij}^k\neq 0$. This gives
        \[\dim(T(G))=2p^r-1+2p^r-2+(p^r-2)(p^r-1)=p^{2r}+p^r-1. \qedhere\]
        \end{proof}

        \begin{Thm}\label{cor:p2dim}
            Let the Camina $p-$group $G$ have nilpotency class $2$, where $|G|=p^n$ and $|Z(G)|=p^k$. Then $\dim T(G)=(p^{n-k}+p^k-1)(p^{n-k}+p^k-2)+p^{n}$.
        \end{Thm}

\begin{proof}
By Theorems \ref{thm:tanaka} and \ref{thm:acclassify}, $T(G)$ is AC and triply regular. Therefore, $\dim T(G)=|\{(i,j,h)\colon p_{ij}^h\neq 0\}|$. As $G$ has nilpotency class $2$, its lower central series is $\{e\}\triangleleft Z(G)=G'\triangleleft G$.

    As $G$ is a Camina group the conjugacy classes of $G$ are those inside $G'=Z(G)$ and the non-trivial cosets of $G/G'$. Then there are $p^{n-k}-1+p^k$ classes. We denote these classes as $C_i$. For fixed $i,h$ we find the number of nonzero $p_{ij}^h$.\\
    {\bf Case 1: }$i\neq h$. As $T(G)$ is AC, by Theorem \ref{thm:tanaka} there is a unique $j$ such that $p_{ij}^h\neq 0$. We have $p^{n-k}+p^k-1$ choices for $C_i$ and $p^{n-k}+p^k-2$ choices for $C_h$, resulting in $(p^{n-k}+p^k-1)(p^{n-k}+p^k-2)$ triples $(i,j,h)$ with $p_{ij}^h\neq 0$.\\
    {\bf Case 2: }$i=h$ and $C_i=\{g\}\subseteq Z(G)$. We have $p_{ij}^i\neq 0$ if and only if $C_i\subseteq C_iC_j$ if and only if $e\in C_j$, if and only if $C_j=\{e\}$, if and only if $j=0$. We have $p^k$ choices for $i=h$, so there are a total of $p^k$ triples $(i,j,h)$ with $p_{ij}^h\neq 0$ in this case.\\
    {\bf Case 3: }$i=h$ with $C_i\not\subseteq Z(G)$. Then $C_i=gZ(G)$, $g\not\in Z(G)$. For $C_j\not\subseteq Z(G)$ we have $C_j=xZ(G)$ for $x\not\in Z(G)$. Then $\overline{C_i}\cdot\overline{C_j}=|Z(G)|\overline{gxZ(G)}$. As $g\in gZ(G)$, for $p_{ij}^h\neq 0$ we must have $gxZ(G)=gZ(G)$, which is false. Thus, $p_{ij}^h=0$ for all $C_j\not\subseteq Z(G)$. For any $C_j=\{x\}\subseteq Z(G)$, $\overline{C_i}\cdot\overline{C_j}=\overline{gxZ(G)}=\overline{gZ(G)}=\overline{C_h}$. So $p_{ij}^h=1$ in this case. Hence, for any $C_j\subseteq Z(G)$ we have $p_{ij}^h\neq 0$. We have $p^k$ choices for $C_j$ for each $C_i=C_h$. There are $p^{n-k}-1$ $C_i$ and so there are $p^k(p^{n-k}-1)$ triples $(i,j,h)$ such that $p_{ij}^h\neq 0$ in this case.

    Having considered every possible case we obtain:
    \[\dim T(G)=(p^{n-k}+p^k-1)(p^{n-k}+p^k-2)+p^k+p^{k}(p^{n-k}-1)=(p^{n-k}+p^k-1)(p^{n-k}+p^k-2)+p^n.\qedhere\]
\end{proof}

By a similar argument we have:

\begin{Thm}\label{cor:p3dim}
    Let the Camina $p-$group $G$ have nilpotency class $3$. Suppose $|Z(G)|=p^k$. From Proposition \ref{thm:Camina3}, we have $|G\colon G'|=p^{2n}$, and $|G'\colon Z(G)|=p^n$ for some even integer $n$. Then 
    \[\dim T(G)=3p^{k+n}+3p^{k+2n}+p^{2k}-6p^k+p^{4n}+3p^{3n}-5p^{2n}-6p^n+7.\]
\end{Thm}

Full details of the proof of Theorem \ref{cor:p3dim} can be found in \cite{Bastiandiss}.

\section{Wreath Products and Wedderburn Decompositions}
    We now determine how to write the group association scheme of each of the groups in Theorem \ref{thm:acclassify} as a wreath product of association schemes that are either $1-$class association schemes, or the group association scheme of an abelian group.

    For an  abelian group there is nothing to show, so let $G=\Z_p^r\rtimes Z_{p^r-1}$ for some prime $p$ and $r>0$. We first find the adjacency matrices of $\mathcal{G}(G)$.

    \begin{Lem}\label{lem:frobA1}
    Let $G$ be the Frobenius group $\Z_{p}^r\rtimes \Z_{p^r-1}$ for prime $p$ and $r>0$. Suppose the classes of $G$ are $C_0=\{e\},\ C_1=\Z_{p}^r\setminus \{e\}$, and $C_i=z^{i-1}\Z_p^r$ for $2\leq i\leq p^{r}$ where $z$ is a generator for $\Z_{p^r-1}$ in $G$. Then the $C_i,C_i$ block of $A_1$ is $J_{|C_i|}-I_{|C_i|}$, the $C_0,C_1$ and $C_1,C_0$ blocks are all $1$'s, and for $i,j\geq 1$ the $C_i,C_j$ ($i \neq j$) block of $A_1$ is all $0$'s. 
\end{Lem}

\begin{proof}
    Recall that $G'=\Z_p^r$. First we consider the $C_i,C_i$ block of $A_1$. Say $g,h\in C_i$. Then $h=xgx^{-1}$, $x\in G$. We have $(A_1)_{gh}=1$ if and only if $hg^{-1}\in G'\setminus\{e\}$. Notice that $hg^{-1}=xgx^{-1}g^{-1}=[x^{-1},g^{-1}]\in G'$. Hence, $(A_1)_{gh}=1$ if and only if $hg^{-1}\neq e$ if and only if $h\neq g$. Therefore, in the $C_i,C_i$ block $(A_1)_{gh}=1$ for all non-diagonal entries and $0$ for the diagonal entries. Thus, the $C_i,C_i$ block of $A_1$ is $J_{|C_i|}-I_{|C_i|}$. 

    Now consider the $C_0,C_1$ and $C_1,C_0$ blocks. For $g\in C_1$, $(A_1)_{eg}=1$. Also as $C_1=G'\setminus \{e\}$ with $g\in C_1$, we have $g^{-1}\in C_1$. Then $(A_1)_{ge}=1$, so the $C_0,C_1$ and $C_1,C_0$ blocks of $A_1$ are all $1$ blocks.

    Next we consider the $C_i,C_j$ block of $A_1$ where $i\neq j$ and $i,j\geq 1$. Suppose $g\in C_i$ and $h\in C_j$ with $(A_1)_{gh}=1$. Then $hg^{-1}\in G'\setminus\{e\}\subset G'$. As $i\neq j$ with $i,j\geq 1$ at least one of $i$ and $j$ is greater than $1$. First suppose $i>1$. Then $C_i=z^{i-1}G'$ and $hg^{-1}\in hG'z^{1-i}=hz^{1-i}G'$. However, $hg^{-1}\in G'$ by assumption. Then $hz^{1-i}G'=G'$, so $hz^{1-i}=k$ for some $k\in G'$. This implies $h=kz^{i-1}$. So $h\in G'z^{i-1}=z^{i-1}G'=C_i$. However, $h\in C_j\neq C_i$, a contradiction. Next suppose, $j>1$. Then $C_j=z^{j-1}G'$, so $hg^{-1}\in z^{j-1}G'g^{-1}=z^{j-1}g^{-1}G'$. However, $hg^{-1}\in G'$ by assumption, so $z^{j-1}g^{-1}G'=G'$. This is only true if $z^{j-1}g^{-1}=t\in G'$, so $g\in z^{j-1}G'=C_j$. However, $g\in C_i\neq C_j$, a contradiction. Therefore, $(A_1)_{gh}\neq 1$ for $g\in C_i$, $h\in C_j$. So, for $i,j\geq 1$ the $C_i,C_j$ ($i \neq j$) block of $A_1$ is all $0$'s.   
\end{proof}

\begin{Lem}\label{lem:frobA2}
     Let $G$ be the Frobenius group $\Z_{p}^r\rtimes \Z_{p^r-1}$. Suppose the classes of $G$ are $C_0=\{e\},\ C_1=\Z_{p}^r\setminus \{e\}$, and $C_i=z^{i-1}\Z_p^r$ for $2\leq i\leq p^{r}$, $\Z_{p^{r}-1}=\langle z\rangle$. Then for $i\geq 2$, the $C_u,C_v$ block of $A_i$ is either all $0$'s or all $1$'s. Moreover, the $C_u,C_v$ block of $A_i$ is all $1$'s if and only if $C_vC_{u}^*=C_i$. 
\end{Lem}

\begin{proof}
     Recall $G'=\Z_p^r$. Since $i\geq 2$, $C_i=z^{i-1}G'$.
     First consider $u=v$. Let $x,y\in C_u$. If $x=y$, then $(A_0)_{xy}=1$. If $x\neq y$, then $(A_1)_{xy}=1$ by Lemma \ref{lem:frobA1}. As no two adjacency matrices are nonzero in the same position, $(A_i)_{xy}=0$ for all $x,y\in C_u$. Hence, the $C_u,C_u$ block of $A_i$ is all $0$'s. Now we assume $C_u\neq C_v$. Let $x\in C_u$, $y\in C_v$. We have $(A_i)_{xy}=1$ if and only if $yx^{-1}\in z^{i-1}G'$. However, as $C_u\neq C_v$, by Theorem \ref{thm:groupequiv}, $(y)^G(x^{-1})^G=(yx^{-1})^G$. Therefore, $yx^{-1}\in C_i$ if and only if $C_vC_{u}^*=C_i$. As we chose $x\in C_u$ and $y\in C_v$ arbitrarily, $(A_i)_{xy}=1$ if and only if $C_vC_u^*=C_i$ for all $x\in C_u$, $y\in C_v$. Therefore, the $C_u,C_v$ block of $A_i$ is all $1$'s if and only if $C_vC_u^*=C_i$. 
\end{proof}

Now we can express the matrices as Kronecker products:

\begin{Thm}\label{thm:frobA}
    Let $G$ be the Frobenius group $\Z_{p}^r\rtimes \Z_{p^r-1}$ for some prime $p$ and $r>0$. Suppose the classes of $G$ are $C_0=\{e\},\ C_1=\Z_{p}^r\setminus \{e\}$, and $C_i=z^{i-1}\Z_p^r$ for $2\leq i\leq p^{r}$ where $z$ is a generator for $\Z_{p^r-1}$ in $G$. Then the adjacency matrices of the group association scheme of $G$ are of the form $A_0=I_{p^{r}-1}\otimes I_{p^r},\ A_1=I_{p^r-1}\otimes (J_{p^r}-I_{p^r})$, and for $i\geq 2$, $A_i=Q_{p^r-1}^{i-1}\otimes J_{p^r}$ where 
    \begin{equation}\label{eq:wr1}
    Q_{p^r-1}=\begin{pmatrix}
        0 & 1 & 0 & \cdots & 0 \\
        0 & 0 & 1 & \cdots & 0 \\
        \vdots & \vdots & \ddots & \ddots & \vdots \\
        0 & 0 & 0 & \cdots & 1 \\
        1 & 0 & 0 & \cdots & 0
    \end{pmatrix}.
    \end{equation}
\end{Thm}

\begin{proof}
    The result is clear for $A_0=I_{|G|}=I_{(p^r-1)p^r}$. From Lemma \ref{lem:frobA1}, the $C_i,C_i$ block of $A_1$ is $J_{|C_i|}-I_{|C_i|}$ for all $i$; the $C_0,C_1$ and $C_1,C_0$ blocks are all $1$'s, and for $i,j\geq 1$ with $i\neq j$ the $C_i,C_j$ block is all $0$'s. We shall now consider the blocks of $A_1$ as the cosets of $\Z_p^r$, written as $z^i\Z_p^r$. Note the only difference between this and using conjugacy classes is that $C_0$ and $C_1$ are combined into one block. We also are not reordering the elements. So for all $i\geq 1$ the $z^i\Z_p^r,z^i\Z_p^r$ block of $A_1$ is $J_{p^r}-I_{p^r}$. Also the $z^i\Z_p^r,z^j\Z_p^r$ block is all $0$ for $i\neq j$ with $i,j\geq 1$. Consider the $\Z_p^r,\Z_p^r$ block. For $g\in \Z_p^r$, $(A_1)_{gg}=0$ since $(A_0)_{gg}=1$. Now for $g,h\in \Z_p^r$ with $g\neq h$, at least one of $g$ and $h$ is not the identity. So either $(a)$ $g\in C_0,\ h\in C_1$, $(b)$ $g\in C_1,\ h\in C_0$, or $(c)$ $g,h\in C_1$. In all three cases $(A_1)_{gh}=1$ as the entry $(g,h)$ is either in the $C_0,C_1$ block, the $C_1,C_0$ block, or the $C_1,C_1$ block. Hence, the $\Z_p^r,\Z_p^r$ block of $A_1$ is all $1$ off the main diagonal and $0$ on the main diagonal. So the $\Z_p^r,\Z_p^r$ block of $A_1$ is $J_{p^r}-I_{p^r}$. For the $\Z_p^r,z^i\Z_p^r$ block, $i\geq 1$, $C_{i+1}=z^i\Z_p^r$, $C_{i+1} C_0^*=C_{i+1}$, and $C_{i+1} C_1^*=C_{i+1}C_1=C_i$. Then by Lemma \ref{lem:frobA2}, $A_{i+1}$ is all $1$'s in the $C_0,C_{i+1}$ and $C_1,C_{i+1}$ block. Since no two adjacency matrices are nonzero in the same position $A_1$ is all $0$'s in the $C_0,C_{i+1}$ and $C_1,C_{i+1}$ blocks. Then the $\Z_p^r,z^i\Z_p^r$ block of $A_1$ is all $0$. For the $z^i\Z_p^r,\Z_p^r$ block, $i\geq 1$, $C_{i+1}^*=z^{p^r-i-2}\Z_p^r$. We have $C_{0} C_{i+1}^*=C_{p^r-i-2}$ and $C_{1} C_{i+1}^*=C_{p^r-i-2}$. By Lemma \ref{lem:frobA2}, $A_{p^r-i-2}$ is all $1$'s in the $C_{i+1},C_{0}$ and $C_{i+1},C_{1}$ block. Since no two adjacency matrices are nonzero in the same position, $A_1$ is all $0$'s in the $C_{i+1},C_{0}$ and $C_{i+1},C_{1}$ blocks. Then the $z^i\Z_p^r,\Z_p^r$ block of $A_1$ is all $0$. Therefore, $A_1$ is a matrix in which the $z^i\Z_p^r,z^i\Z_p^r$ block for all $i$ is $J_{p^r}-I_{p^r}$ and the $z^i\Z_p^r,z^j\Z_p^r$ for $i\neq j$ block is all $0$. There are a total of $p^r-1$ cosets of $\Z_p^r$ in $G$. Hence, $A_1=I_{p^r-1}\otimes (J_{p^r}-I_{p^r})$.

    For $A_i$, $i\geq 2$, we consider the blocks of $A_i$ as the cosets of $\Z_p^r$. We saw when looking at $A_1$ that the $\Z_p^r,z^{i-1}\Z_p^r$ block of $A_i$ is all $1$'s. For any $j\neq i-1$,  $A_{j+1}$ is all $1$'s in the $\Z_p^r,z^j\Z_p^r$ block. Hence, $A_i$ must be all $0$'s in the $\Z_p^r,z^j\Z_p^r$ block for all $j\neq i-1$. The $z^{p^r-i-2}\Z_p^r,\Z_p^r$ block of $A_i$ is all $1$'s (working module $p^r-1$). Also for any $j\neq p^r-i-2$, $A_{p^r-j-2}$ is all $1's$ in the $z^{j}\Z_p^r,\Z_p^r$ block. Hence, $A_i$ must be all $0$'s in the $z^j\Z_p^r,\Z_p^r$ block for $j\neq p^r-i-2$. We now consider the $z^j\Z_p^r,z^k\Z_p^r$ block, $j,k\geq 1$. From Lemma \ref{lem:frobA2}, the $z^j\Z_p^r,z^k\Z_p^r$ block of $A_i$ is all $1$'s if and only if $z^k\Z_p^r\cdot z^{-j}\Z_p^r=z^{i-1}\Z_p^r$ and is all $0$'s otherwise. Hence, the $z^j\Z_p^r,z^k\Z_p^r$ block of $A_i$ is all $1$'s if and only if $z^{k-j}\Z_p^r=z^{i-1}\Z_p^r$. These two cosets are equal if and only if $z^{k-j}=z^{i-1}g$ for some $g\in \Z_p^r$. So $z^{k-j-i+1}=g\in \Z_p^r$. However, the only power of $z$ that is in $\Z_p^r$ is $z^0=e$. Hence, the $z^j\Z_p^r,z^k\Z_p^r$ block of $A_i$ is all $1$'s if and only if $k-j-i+1=0$. That is, if $i-1\equiv k-j \mod p^{r-1}$. Given any fixed $z^j\Z_p^r$ with $j\geq 1$ row block there is then a unique column block $z^{j+i-1}\Z_p^r$ such that the $z^j\Z_p^r,z^{j+i-1}\Z_p^r$ block of $A_i$ is all $1$'s. When $j=0$ the $z^j\Z_p^r,z^{j+i-1}\Z_p^r$ block is just the $\Z_p^r,z^{i-1}\Z_p^r$ block, which is the unique block in the $\Z_p^r$ row block where $A_i$ is nonzero. When $j=p^r-i-2$ the $z^j\Z_p^r,z^{j+i-1}\Z_p^r$ block is just the $z^{p^r-i-2}\Z_p^r,z^{p^r-1}\Z_p^r$ block, which is the same as the $z^{p^r-i-2}\Z_p^r,\Z_p^r$ block. This is the unique block in the $\Z_p^r$ column block where $A_i$ is nonzero. Combining all of this we have for all $j$ that the $z^j\Z_p^r,z^{i+j-1}\Z_p^r$ block of $A_i$ is the all $1$'s matrix and all other blocks are $0$. Thus $A_i=B\otimes J_{p^r}$ where $B$ is the $p^r-1\times p^r-1$ matrix that is nonzero in the $j,i+j-1 \mod p^r-1$ entry for all $j$ and $0$ elsewhere. Namely the nonzero entries of $B$ are $(1,i)$, $(2,i+1)$, $(3,i+2)$, $\cdots,$ $(p^{r}-1,i-1)$. Directly computing powers of $Q_{p^r-1}$ we see that $Q_{p^r-1}^i$ is nonzero in entries $(1,i+1), (2,i+2),\cdots, (p^r-1,i)$. Hence, $B=Q_{p^r-1}^{i-1}$. So for $i\geq 2$, $A_i=Q_{p^r-1}^{i-1}\otimes J_{p^r}$ as claimed. 
\end{proof}

\begin{Cor}\label{cor:frobwreath}
     Let $G$ be the Frobenius group $\Z_{p}^r\rtimes \Z_{p^r-1}$. Then $\mathcal{G}(G)=\mathcal{K}_{p^r}\wr \mathcal{G}(\Z_{p^{r}-1})$. 
\end{Cor}

\begin{proof}
    The adjacency matrices of $K_{p^r}\wr \mathcal{G}(\Z_{p^{r}-1})$ are just $D_0=I_{p^r(p^r-1)}$, $D_1=I_{p^r-1}\otimes (J_{p^r}-I_{p^r})$, and $D_{i+1}=B_{i}\otimes J_{p^r}$ for $1\leq i\leq p^{r}-1$, where the $B_{i}$ are the adjacency matrices for $\mathcal{G}(\Z_{p^r-1})$. We show that $B_{i-1}=Q_{p^r-1}^{i-1}$ for each $i\geq 2$ where $Q_{p^r-1}$ is the matrix in (\ref{eq:wr1}). Let $\Z_{p^r-1}=\langle z\rangle$. We note
    \[(B_i)_{xy}=\left\{\begin{array}{cc}
        1 & \text{ if }yx^{-1}=z^i \\
        0 & \text{ otherwise }
    \end{array}\right..\]
    Ordering the elements of $\Z_{p^r-1}$ as $z^0,z^1,z^2,\cdots, z^{p^r-2}$ the nonzero entries of $B_{i}$ are $(1,i+1),\ (2,i+2),\ \cdots, (p^r-1,i)$. These are the nonzero entries of $Q_{p^r-1}^{i}$, so $B_i=Q_{p^r-1}^{i-1}$. Hence, the adjacency matrices of $K_{p^r}\wr \mathcal{G}(\Z_{p^{r}-1})$ are just $D_0=I_{p^r(p^r-1)}$, $D_1=I_{p^r-1}\otimes (J_{p^r}-I_{p^r})$, and $D_{i+1}=Q_{p^r-1}^{i}\otimes J_{p^r}$ for $1\leq i\leq p^r-1$. Shifting the index on $D_{i+1}$ for $1\leq i\leq p^r-1$ the adjacency matrices of $K_{p^r}\wr \mathcal{G}(\Z_{p^{r}-1})$ are $D_0=I_{p^{r}-1}\otimes I_{p^r}, D_1=I_{p^r-1}\otimes (J_{p^r}-I_{p^r})$, and for $i\geq 2$, $D_i=Q_{p^r-1}^{i-1}\otimes J_{p^r}$. From Theorem \ref{thm:frobA}, these are exactly the adjacency matrices for $\mathcal{G}(G)$. Hence, $\mathcal{G}(G)=K_{p^r}\wr \mathcal{G}(\Z_{p^{r}-1})$.
\end{proof}

For a Camina $p-$group $G$, the wreath product will depend on the nilpotency class of $G$. We begin by finding the adjacency matrices for $\mathcal{G}(G)$.

\begin{Lem}\label{lem:cam2outer}
    Let $G$ be a Camina $p-$group and $C_k\not\subseteq G'$. Then for classes $C_i,C_j$ of $G$ the $C_i,C_j$ block of $A_k$ is all $0$'s if $C_j\not\subseteq C_iC_k$ and all $1$'s if $C_j\subseteq C_iC_k$.
\end{Lem}

\begin{proof}
    As $C_k\not\subseteq G'$, $C_k=gG'$ for some $g\in G\setminus G'$. Now let $x\in C_i$, $y\in C_j$. We have $(A_k)_{xy}=1$ if and only if $yx^{-1}\in gG'$, if and only if $y\in gG'x=gxG'$, if and only if $C_j\subseteq gxG'$. Now either $C_i\subseteq G'$ or $C_i=hG'$ for some $h\in G\setminus G'$. If $C_i\subseteq G'$, then $x\in G'$ and $gxG'=gG'=C_iC_k$. If $C_i=hG'$, $h\in G\setminus G'$, then $C_i=xG'$ as well. Then $gxG'=C_kC_i=C_iC_k$. So either way $C_iC_k=gxG'$. So for $x\in C_i$, $y\in C_j$, $(A_k)_{xy}=1$ if and only if $C_j\subseteq C_iC_k$.
 
    Now suppose the $C_i,C_j$ block of $A_k$ has a nonzero entry: $(A_k)_{ab}=1$ for $a\in C_i, b\in C_j$. Then $ba^{-1}\in C_k$. Hence, $b\in C_ka\subseteq C_iC_k$. This implies $C_j\subseteq C_iC_k$, and so for $x\in C_i$ and $y\in C_j$, $(A_k)_{xy}=1$. Thus, the $C_i,C_j$ block of $A_k$ is all $1$'s, so the $C_i,C_j$ block of $A_k$ is all $0$'s or all $1$'s.
\end{proof}

\begin{Lem}\label{lem:cam3centeral}
    Let $C_k\subseteq Z(G)$ be a class of $G$. If $G$ has nilpotency class $2$, let $C_i,C_j$ be any two classes of $G$. If $G$ has nilpotency class $3$, then suppose $C_i,C_j\subseteq G'$ are classes of $G$. Then
    \begin{enumerate}
        \item If $C_i,C_j\subseteq Z(G)$, then the $C_i,C_j$ block of $A_k$ is $0$ if $C_j\neq C_iC_k$ and $1$ if $C_j=C_iC_k$.
        \item If $C_i\neq C_j$ with either $C_i\not\subseteq Z(G)$ or $C_j\not\subseteq Z(G)$, then the $C_i,C_j$ block of $A_k$ is all $0$.
        \item If $C_i\not\subseteq Z(G)$, then the $C_i,C_i$ block of $A_k$ has a single nonzero entry in each row.
    \end{enumerate}
\end{Lem}

\begin{proof}
    As $C_k\subseteq Z(G)$, $C_k=\{g\}$ for $g\in Z(G)$.\\
    1: If $C_i,C_j\subseteq Z(G)$, then $C_i=\{x\}$ and $C_j=\{y\}$, $x,y\in Z(G)$. Then the $C_i,C_j$ block of $A_k$ is just the $(x,y)$ entry. So $(A_k)_{xy}=1$ if and only if $yx^{-1}\in C_k=\{g\}$. That is, $(A_k)_{xy}=1$ if and only if $yx^{-1}=g$, Hence, the $C_i,C_j$ block of $A_k$ is $0$ if $C_j\neq C_iC_k$ and $1$ if $C_j=C_iC_k$.\\
    2: Assume $C_i\neq C_j$ with either $C_i\not\subseteq Z(G)$ or $C_j\not\subseteq Z(G)$. Without loss of generality suppose $C_i\not\subseteq Z(G)$. Then by Proposition \ref{lem:pcamclass}, $C_i=hZ(G)$ for some $h\in G$. Let $x\in C_i$ and $y\in C_j$. If $(A_k)_{xy}=1$, then $yx^{-1}\in C_k\subseteq Z(G)$. We note as $x\in C_i=hZ(G)$, that $x^{-1}\in h^{-1}Z(G)$. If $C_j\subseteq Z(G)$, then $yx^{-1}\in h^{-1}Z(G)$. However, $h^{-1}Z(G)\cap Z(G)=\emptyset$, which contradicts $yx^{-1}\in h^{-1}Z(G)\cap Z(G)$. If $C_j\subseteq G\setminus Z(G)$, then by Proposition \ref{lem:pcamclass}, $C_j=\ell Z(G)$ for $\ell\in G\setminus Z(G)$. Since $C_i\neq C_j$, $\ell h^{-1}Z(G)\neq Z(G)$. As $y\in \ell Z(G)$ and $x^{-1}\in h^{-1}Z(G)$, $yx^{-1}\in \ell h^{-1}Z(G)$. However, $\ell h^{-1}Z(G)\cap Z(G)=\emptyset$ which contradicts $yx^{-1}\in h^{-1}Z(G)\cap Z(G)$. In either case we got a contradiction. Therefore, $(A_k)_{xy}=0$.\\
    3: Assume $C_i=C_j$ with $C_i\not\subseteq Z(G)$ and fix $x\in C_i$. Since $C_i\not\subseteq  Z(G)$ by Proposition \ref{lem:pcamclass}, $C_i=xZ(G)$. Then $(A_k)_{xy}=1$ if and only if $yx^{-1}\in C_k=\{g\}$, if and only if $y=xg$. Now as $C_i=xZ(G)$, we have $xg\in C_i$. So the $(x,xg)$ entry of $A_k$ is nonzero. This however is the only nonzero entry of the $x$th row of $A_k$. As we chose $x\in C_i$ arbitrarily, the $C_i,C_i$ block of $A_k$ has a single nonzero entry in each row. 
\end{proof}

For $G$ of nilpotency class $2$ we now write the adjacency matrices as Kronecker products.

\begin{Lem}\label{lem:cam2class}
 Let $G$ be a Camina $p-$group of nilpotency class $2$. Suppose $|Z(G)|=p^r$ and $|G|=p^n$. If $C_k\subseteq Z(G)$ is a conjugacy class of $G$, $C_k=\{g\}$, then $A_k=I_{p^{n-r}}\otimes B$ where $B$ is the adjacency matrix of \{g\} when considered as a conjugacy class in $Z(G)$. 
\end{Lem}

\begin{proof}
    The blocks of $A_k$ correspond to cosets of $Z(G)=G'$, so the $xZ(G),yZ(G)$ block of $A_k$ is the same as the $x^G,y^G$ block for any $x,y\not\in Z(G)$. So in this situation the blocks are the same as the class blocks. The only difference is that classes in $Z(G)$ are now being considered as a single block. From Lemma \ref{lem:cam3centeral}, the $xZ(G),yZ(G)$ block of $A_k$ is all $0$ if $xZ(G)\neq yZ(G)$ for any $x,y\in G$. Hence, only the diagonal blocks are nonzero. Additionally from the proof of Lemma \ref{lem:cam3centeral}, in the $xZ(G),xZ(G)$ block for $x\not\in Z(G)$ the only nonzero entry in the $y$th row of $A_k$ is the $(y,gy)$ entry for all $y\in xZ(G)$. Now consider the $Z(G),Z(G)$ block. Consider row $y\in Z(G)$ of this block. From Lemma \ref{lem:cam3centeral}, the $(y,z)$ entry of $A_k$ is $1$ if and only if $z^G=g^Gy^G$. However, as all of $y,g,z\in Z(G)$ this is the same thing as saying $z=gy$. Hence, in the $Z(G),Z(G)$ block of $A_k$ an entry is nonzero if and only if it is the $(y,gy)$ entry for some $y\in Z(G)$. Hence, for all $x\in G$ the $xZ(G),xZ(G)$ block of $A_k$ has a nonzero entry in entry $(y,gy)$ for all $y\in xZ(G)$ and is $0$ otherwise. Additionally, $A_k$ is $0$ outside of these diagonal blocks. Then $A_k=I_{p^{n-r}}\otimes B$ where $B$ is a $p^r\times p^r$ matrix such that for row $y$, the only nonzero entry is $(y,gy)$. That is, $B$ is a $p^r\times p^r$ matrix such that $(B)_{uv}=1$ if and only if $vu^{-1}=g$. This is identical to the condition for the adjacency matrix of $\{g\}$ in $Z(G)$ to have a nonzero entry. Hence, $B$ is the adjacency matrix of $\{g\}$ when considered as a conjugacy class in $Z(G)$.   
\end{proof}

\begin{Lem}\label{lem:cam2class'}
 Let $G$ be a Camina $p-$group of nilpotency class $2$ for some prime $p$. Suppose $|Z(G)|=p^r$ and $|G|=p^n$. If $C_k$ is a conjugacy class of $G$ such that $C_k\not\subseteq Z(G)$, with $C_k=gZ(G)$, then $A_k=D\otimes J_{p^r}$ where $D$ is the adjacency matrix of the conjugacy class $\{gZ(G)\}$ in $G/Z(G)$.
\end{Lem}

\begin{proof}
    The proof is similar to that of Lemma \ref{lem:cam2class}, using Lemma \ref{lem:cam2outer}.
\end{proof}

Now that we have determined the adjacency matrices of $\mathcal{G}(G)$ we can write $\mathcal{G}(G)$ as a wreath product of two group association schemes of abelian groups.

\begin{Cor}\label{cor:cam2wreath}
     Let $G$ be a Camina $p-$group of nilpotency class $2$ for some prime $p$, where $|Z(G)|=p^r$ and $|G|=p^n$. Then $\mathcal{G}(G)=\mathcal{G}(Z(G))\wr \mathcal{G}(G/Z(G))$.
\end{Cor}

\begin{proof}
    Let $A_0,A_1,\cdots, A_{p^r}$ be the adjacency matrices of $\mathcal{G}(Z(G))$ and $B_0,B_1,\cdots, B_{p^{n-r}}$ be the adjacency matrices of $\mathcal{G}(G/Z(G))$. The adjacency matrices of $\mathcal{G}(Z(G))\wr \mathcal{G}(G/Z(G))$ are $D_0=I_{p^{n-r}}\otimes I_{p^r}=I_{p^n}$, $D_k=I_{p^{n-r}}\otimes A_k$ where $1\leq k\leq p^r$, and $D_{p^r+k}=B_k\otimes J_{p^r}$ for $1\leq k\leq p^{n-r}$.  From Lemma \ref{lem:cam2class}, all the adjacency matrices $M_k$ of $\mathcal{G}(G)$ for some class $C_k=\{g\}\subset Z(G)$ are of the form $M_k=I_{p^{n-r}}\otimes A_k$ where $A_k$ is the adjacency matrix for $C_k=\{g\}$ in $Z(G)$. Ranging over all $C_k\subset Z(G)$ we have $D_i=M_i$ for all $0\leq i\leq p^r$. From Lemma \ref{lem:cam2class'}, the adjacency matrices $M_k$ of $\mathcal{G}(G)$ for any class $C_k=gZ(G)$, has $M_k=B_\ell\otimes J_{p^r}$ where $B_\ell$ is the adjacency matrix for $\{gZ(G)\}$ in $G/Z(G)$. We then have a one-to-one correspondence between the adjacency matrices $M_k$ of $\mathcal{G}(G)$ for $C_k\not\subseteq Z(G)$ and the adjacency matrices of the form $D_{p^r+k}=B_k\otimes J_{p^r}$ for $\mathcal{G}(Z(G))\wr \mathcal{G}(G/Z(G))$.  

    We then see $\mathcal{G}(G)$ and $\mathcal{G}(Z(G))\wr \mathcal{G}(G/Z(G))$ have the exact same adjacency matrices. Therefore, $\mathcal{G}(G)=\mathcal{G}(Z(G))\wr \mathcal{G}(G/Z(G))$.    
\end{proof}

For $G$ a Camina $p-$group with nilpotency class $3$ a similar argument is followed to get the wreath product decomposition. We state it here. A full proof can be found in \cite{Bastiandiss}.

\begin{Cor}\label{cor:cam3wreath}
    Let $G$ be a Camina $p-$group of nilpotency class $3$ for some prime $p$. Suppose that $|Z(G)|=p^r$ for some $r>0$ and $|G'\colon Z(G)|=p^n$. Then $\mathcal{G}(G)=\mathcal{G}(Z(G))\wr \mathcal{G}(G'/Z(G))\wr \mathcal{G}(G/G')$.
\end{Cor}

The last group to consider from Theorem \ref{thm:acclassify}, is $G=\Z_3^2\rtimes Q_8$. The wreath product decomposition can be found directly using Magma \cite{Magma}:

\begin{Prop}\label{prop:72,41wreath}
    Let $G=\Z_3^2\rtimes Q_8$. Then $\mathcal{G}(G)=\mathcal{K}_9\wr \mathcal{G}(\Z_2)\wr \mathcal{G}(\Z_2^2)$. \hfill$\qed$
\end{Prop}

We now give the Wedderburn components for the Terwilliger algebra in each of these cases. The actual computation of these Wedderburn components can be done using Theorem 5.3 of \cite{WreathProduct1}. We note in doing so, the definition of the wreath product we have used differs from that in \cite{WreathProduct1}. However, they are equivalent via a permutation matrix. Full details of the computation of these idempotents can be found in \cite{Bastiandiss}.

First we note that the Wedderburn decomposition for $T(G)$ when $G$ is an abelian group only consists of the primary component $V$, with dimension $|G|^2$. Now we state the Wedderburn decomposition in the other cases. Using Theorem 5.3 of \cite{WreathProduct1} we have:

\begin{Thm}\label{cor:frobdecomp}
     Let $G=Z_p^{r}\rtimes Z_{p^r-1}$ for some prime $p$ and $n\geq 1$. Suppose $C_0=\{e\},C_1,\cdots, C_{p^r-1}$ are the conjugacy classes of $G$. Then the Wedderburn decomposition of $T(G)$ is
     \[T(G)\cong V\oplus W_1\oplus W_2 \oplus \cdots \oplus W_{p^r-1}\]
     where each $W_i=\emph{Span}(B_i)$ and the idempotent $B_i$ is defined to be the $|G|\times |G|$ matrix indexed by $G$ which is $0$ outside the $C_i,C_i$ block and has $C_i,C_i$ block
    \[d_{ii}=\frac{-1}{|C_i|-1}J_{|C_i|}+\left(1-\frac{-1}{|C_i|-1}\right)I_{|C_i|}=\begin{pmatrix}
        1 & \frac{-1}{|C_i|-1} & \frac{-1}{|C_i|-1} & \cdots & \frac{-1}{|C_i|-1} \\ \frac{-1}{|C_i|-1} & 1 & \frac{-1}{|C_i|-1} & \cdots & \frac{-1}{|C_i|-1} \\
        \vdots & & \ddots & & \vdots \\
        \vdots & & & \ddots & \vdots \\
        \frac{-1}{|C_i|-1} & \frac{-1}{|C_i|-1} & \cdots & \frac{-1}{|C_i|-1} & 1
    \end{pmatrix}.\]
\end{Thm}

\begin{Thm}\label{thm:cam2decomp}
    Let $G$ be a Camina $p-$group with nilpotency class $2$ of order $p^n$, and $|Z(G)|=p^k$. Then the Wedderburn decomposition of $T(G)$ is
    \[T(G)=V\oplus \bigoplus_{i=1}^{p^{n-k}-1}\bigoplus_{\alpha}\mathcal{W}_i(\alpha),\]
    where $V$ is the primary component and each $\mathcal{W}_i(\alpha)$ is the one-dimensional ideal generated by a matrix 
    \[W_i(\alpha)=\frac{1}{p^k}\sum_{a_1,a_2,\cdots, a_k\in \{0,1,\cdots, p-1\}} \prod_{j=1}^k \zeta^{\alpha_j a_j}(E_i^*A_jE_i^*)^{a_j},\] 
    where $\alpha=(\alpha_1,\alpha_2,\cdots, \alpha_k)$ with each $\alpha_j\in \{0,1,2,\cdots, p-1\}$ and not all $\alpha_j=0$.
\end{Thm}

\begin{Thm}\label{thm:cam3decomp}
    Let $G$ be a Camina $p-$group with nilpotency class $3$, $|G:G'|=p^{2n}$, $|G':Z(G)|=p^n$ for some even integer $n$. Suppose $|Z(G)|=p^k$. Then the Wedderburn decomposition of $T(G)$ is
    \[T(G)=V\oplus \bigoplus_{i=1}^{p^{n}-1}\bigoplus_\alpha\mathcal{W}_i(\alpha)\oplus \bigoplus_{t=1}^{p^{2n}-1}\bigoplus_\beta \mathcal{X}_t(\beta)\oplus \bigoplus_{s=1}^{p^{2n}-1} \bigoplus_\gamma \mathcal{Y}_s(\gamma),\]
    where $V$ is the primary component, each $\mathcal{W}_i(\alpha)$ is a one-dimensional ideal generated by
    \[W_i(\alpha)=\frac{1}{p^k}\sum_{a_1,a_2,\cdots, a_k\in \{0,1,\cdots, p-1\}} \prod_{j=1}^k \zeta^{\alpha_j a_j}(E_i^*A_jE_i^*)^{a_j},\]
    where $\alpha=(\alpha_1,\alpha_2,\cdots, \alpha_k)$ with each $\alpha_j\in \{0,1,2,\cdots, p-1\}$ not all $\alpha_j$ are $0$. Each $\mathcal{X}_t(\beta)$ is a one-dimensional ideal generated by 
     \[X_t(\beta)=\frac{1}{p^k}\sum_{a_1,a_2,\cdots, a_k\in \{0,1,\cdots, p-1\}} \prod_{j=1}^k \zeta^{\beta_j a_j}(E_t^*A_jE_t^*)^{a_j}\]
     where $\beta=(\beta_1,\beta_2,\cdots, \beta_k)$ with each $\beta_j\in \{0,1,\cdots, p-1\}$ not all $\beta_j$ are zero. Each $\mathcal{Y}_s(\gamma)$ is a one-dimensional ideal generated by 
    \[Y_s(\gamma)=\frac{1}{p^{n+k}}\left(\sum_{j=0}^{p^k-1} E_s^*A_jE_s^*+\sum_{a_1,a_2,\cdots, a_n} p^k\prod_{j=1}^n \frac{\zeta^{\gamma_j a_j}}{p^{a_jk}}(E_s^*A_{p^k-1+j}E_s^*)^{a_j}\right)\]
    where $\gamma=(\gamma_1,\gamma_2,\cdots, \gamma_n)$ with each $\gamma_j\in \{0,1,\cdots, p-1\}$ not all $\gamma_j$ are zero. 
\end{Thm}

\begin{Thm}\label{cor: 72,41 decomp}
            Let $G=\Z_3^2\rtimes Q_8$. Then the Wedderburn decomposition of $T(G)$ is
            \[T(G)\cong V\oplus Z_1\oplus Z_2 \oplus Z_3 \oplus \cdots \oplus Z_8 \]
            where each $Z_i$, $1\leq i\leq 8$ is a $1-$dimensional ideal of $T(G)$ and $V$ is the primary component. 
        \end{Thm}
The exact form of the $Z_i$ in Theorem \ref{cor: 72,41 decomp} can be found as an easy calculation using \cite{Magma}.

\printbibliography

@article{Ma89,
	author = {Siu Lun Ma},
	title 	= {On association schemes, {Schur} rings, strongly regular graphs and partial difference sets},
	year	= {1989},
	journal	= {Ars Combinatoria},
	volume	= {27},
	pages	= {211--220}
}

@article{Wielandt49,
	author = {Helmut Wielandt},
	title 	= {Zur Theorie der einfach transitiven Permutationsgruppen {II}},
    year 	= {1949},
	journal = {Math. Z.},
	volume = {52},
	pages	= {384--393},
    notes = {German}
}

@article{Terwilliger1,
	author = {Paul Terwilliger},
	title 	= {The subconstituent algebra of an association scheme, {(Part I)}},
    year 	= {1992},
	journal = {Journal of Algebraic Combinatorics},
	volume = {1},
	pages	= {363--388}
}

@article{Terwilliger2,
	author = {Paul Terwilliger},
	title 	= {The subconstituent algebra of an association scheme, {(Part II)}},
    year 	= {1993},
	journal = {Journal of Algebraic Combinatorics},
	volume = {2},
	pages	= {73--103}
}

@article{Terwilliger3,
	author = {Paul Terwilliger},
	title 	= {The subconstituent algebra of an association scheme, {(Part III)}},
    year 	= {1993},
	journal = {Journal of Algebraic Combinatorics},
	volume = {2},
	pages	= {177--210}
}

@article{Bannaiarticle,
	author = {Eiichi Bannai and Akihiro Munemasa},
	title 	= {The {Terwilliger} algebras of group association schemes},
    year 	= {1995},
	journal = {Kyushu Journal of Mathematics},
	volume = {49},
	pages	= {93--102}
}

@article{WreathProduct1,
	author = {Mikhail Muzychuck and Bangteng Xu},
	title 	= {Terwilliger algebras of wreath products of association schemes},
    year 	= {2016},
	journal = {Linear Algebra and its Applications},
	volume = {493},
	pages	= {146--163}
}

@article{GroupAssoc2,
	author = {Nur Hamid and Manabu Oura},
	title 	= {Terwilliger algebras of some group association schemes},
    year 	= {2019},
	journal = {Math. J. Okayama Univ.},
	volume = {61},
	pages	= {199--204}
}

@article{GroupAssoc1,
	author = {Jose Maria P. Balmaceda and Manabu Oura},
	title 	= {The {Terwilliger} algebras of the group association schemes of {$S_5$} and {$A_5$}},
    year 	= {1994},
	journal = {Kyushu Journal of Mathematics},
	volume = {48},
	pages	= {221--231}
}

@article{Tanaka,
    author = {Rie Tanaka},
    title = {Classification of commutative association schemes with almost commutative {Terwilliger} algebras},
    journal = {Journal of Algebraic Combinatorics},
    volume = {33},
    pages = {1-10},
    year = {2011}
}

@article{Lewis1,
    author = {Lewis, Mark},
    title = {{Camina} groups, {Camina} pairs, and generalizations},
    journal = {Group Theory and Computation (Indian Statistical Institute Series)},
    pages = {141-173},
    year = {2018}
}

@article{Few_Conj,
    author = {Leonardo Cangelmi and Arun S. Muktibodh},
    title = {Camina groups with few conjugacy classes},
    journal = {Algebra Discrete Math},
    volume = {9},
    pages = {39-49},
    year = {2010}
}

@article{Product_Conj,
    author = {Everett C. Dade and Manoj K. Yadav},
    title = {Finite groups with many product conjugacy classes},
    journal = {Israel Journal of Mathematics},
    volume = {154},
    pages = {29-49},
    year = {2006}
}

@article{Dark1,
    author = {Rex Dark and Carlo M. Scoppola},
    title = {On {Camina} groups of prime power order},
    journal = {Journal of Algebra},
    volume = {181},
    pages = {787-802},
    year = {1996}
}

@article{Camina,
    author = {A.R. Camina},
    title = {Some conditions which almost characterize {Frobenius} groups},
    journal = {Israel Journal of Mathematics},
    volume = {31},
    pages = {153-160},
    year = {1978}
}

@article{MacDonald,
    author = {I.D. MacDonald},
    title = {Some p-groups of {Frobenius} and extra-special type},
    journal = {Israel Journal of Mathematics},
    volume = {40},
    pages = {350-364},
    year = {1981}
}

@unpublished{Bastian,
    author = {Nicholas Lee Bastian},
    title = {{Terwilliger algebras for several finite groups}},
    note = {Master's Thesis, Brigham Young University 2021}
}

@article {Magma,
    AUTHOR = {Bosma, Wieb and Cannon, John and Playoust, Catherine},
     TITLE = {The {M}agma algebra system. {I}. {T}he user language},
      NOTE = {Computational algebra and number theory (London, 1993)},
   JOURNAL = {J. Symbolic Comput.},
  FJOURNAL = {Journal of Symbolic Computation},
    VOLUME = {24},
      YEAR = {1997},
    NUMBER = {3-4},
     PAGES = {235--265},
      ISSN = {0747-7171},
   MRCLASS = {68Q40},
  MRNUMBER = {MR1484478},
       DOI = {10.1006/jsco.1996.0125},
       URL = {http://dx.doi.org/10.1006/jsco.1996.0125},
}

@phdthesis{Bastiandiss,
    author = {Bastian, Nicholas Lee},
    title = {Almost commutatitve {Terwilliger} algebras for {Schur} rings},
    school = {Brigham Young University},
    year = {2026}
}

@unpublished{Bastiansgp,
    author = {Bastian, Nicholas Lee and Stephen P. Humphries},
    title = {Almost Commutative {Terwilliger} Algebras II: Strong {Gelfand} Pairs},
    note = {(preprint), 17 pages},
    year = {2025}
}

\end{document}